\documentclass[a4paper,11pt]{amsart}
\usepackage{amssymb}
\usepackage{amsmath}
\usepackage{amsthm}
\usepackage{hyperref}
\usepackage{cleveref}
\usepackage{verbatim}
\usepackage{url}
\usepackage[all]{xy}

\usepackage{color}

\theoremstyle{plain}

\newtheorem{theorem}{Theorem}[section]
\crefname{section}{Section}{}
\crefname{theorem}{Theorem}{}

\newtheorem{proposition}[theorem]{Proposition}
\crefname{proposition}{Proposition}{}
\newtheorem{lemma}[theorem]{Lemma}
\crefname{lemma}{Lemma}{Lemmas}
\newtheorem{corollary}[theorem]{Corollary}

\theoremstyle{definition}

\newtheorem{subsec}[theorem]{}
\crefname{subsec}{Subsection}{}
\newtheorem{remark}[theorem]{Remark}
\crefname{remark}{Remark}{}
\newtheorem{definition}[theorem]{Definition}
\crefname{definition}{Definition}{}
\newtheorem{example}[theorem]{Example}
\crefname{example}{Example}{}

\newtheorem*{example*}{Example}

\newcommand{\R}{{\mathbb{R}}}
\newcommand{\Q}{{\mathbb{Q}}}
\newcommand{\Z}{{\mathbb{Z}}}
\newcommand{\G}{{\mathbb{G}}}

\newcommand{\sX}{{\mathcal X}}
\newcommand{\sV}{{\mathcal V}}
\newcommand{\sP}{{\mathcal P}}
\newcommand{\sD}{{\mathcal{D}}}
\newcommand{\sC}{{\mathcal C}}
\newcommand{\sF}{{\mathcal F}}
\newcommand{\sA}{{\mathcal A}}
\newcommand{\sN}{{\mathcal{N}}}

\newcommand{\sE}{{\mathcal{E}}}

\newcommand{\vs}{\varsigma}
\newcommand{\ve}{{\varepsilon}}

\newcommand{\vk}{\varkappa}

\newcommand{\into}{\hookrightarrow}
\newcommand{\isoto}{\overset{\sim}{\to}}

\newcommand{\labelto}[1]{\xrightarrow{\makebox[1.5em]{\scriptsize ${#1}$}}}

\newcommand{\charr}{{\rm char\,}}
\newcommand{\Hom}{{\rm Hom}}
\newcommand{\Spec}{{\rm Spec\,}}
\newcommand{\Ker}{{\rm Ker\,}}
\newcommand{\Span}{{\rm Span\,}}
\renewcommand{\dim}{{\rm dim}}
\newcommand{\SAut}{{\rm SAut}}
\newcommand{\Aut}{{\rm Aut}}
\newcommand{\Val}{{\rm Val}}
\newcommand{\val}{{\rm val}}
\newcommand{\Gal}{{\rm Gal}}
\newcommand{\Stab}{{\rm Stab}}

\newcommand{\siga}{{\sigma_\gamma}}
\newcommand{\muga}{{\mu_\gamma}}
\newcommand{\ega}{{\ve_\gamma}}

\newcommand{\upgam}{\hs^\gamma\!}
\newcommand{\qs}{{\rm qs}}
\newcommand{\ad}{{\rm ad}}

\newcommand{\Gbar}{{\overline{G}}}
\newcommand{\Ztil}{{\widetilde{Z}}}

\newcommand{\hs}{\kern 0.75pt}

\newcommand{\X}{{{\sf X}}}

\newcommand{\inn}{\mathrm{inn}}
\newcommand{\Inn}{\mathrm{Inn}}

\newcommand{\BRD}{{\rm BRD}}

\newcommand{\id}{{\rm id}}

\newcommand{\veg}{{\ve_\gamma}}

\newcommand{\Gtil}{{\widetilde{G}}}

\newcommand{\SL}{{\rm SL}}

\newcommand{\Lie}{{\rm Lie\,}}

\renewcommand{\AA}{{\sf A}}

\newcommand{\DD}{{\sf D}}

\newcommand{\Spin}{{\rm Spin}}

\newcommand{\vktil}{{\widetilde \vk}}

\title{Spherical varieties over large fields}

\author{Stephan Snegirov}

\address{Raymond and Beverly Sackler School of Mathematical Sciences,
Tel Aviv University, 6997801 Tel Aviv, Israel}

\email{stephans@mail.tau.ac.il}

\thanks{This research was partially supported by the Israel Science Foundation (grant No. 870/16)}



\date{\today}

\begin{document}

\begin{abstract}
Let $k_0$ be a field of characteristic 0, $k$ its algebraic closure,  $G$ a connected reductive group defined over $k$.
Let $H\subset G$ be  a spherical subgroup. We assume that $k_0$ is a large field, for example,
$k_0$ is either the field $\R$ of real numbers or a $p$-adic field.
Let $G_0$ be a quasi-split $k_0$-form of $G$.
We show that if $H$ has self-normalizing normalizer, and $\Gamma = \Gal(k/k_0)$ preserves the combinatorial invariants of $G/H$,
then $H$ is conjugate to a subgroup defined over $k_0$, and hence, the $G$-variety $G/H$ admits a $G_0$-equivariant $k_0$-form.
In the case when $G_0$ is not assumed to be quasi-split, we give a necessary and sufficient Galois-cohomological condition
for the existence of a $G_0$-equivariant $k_0$-form of $G/H$.

\end{abstract}

\maketitle

\section{Introduction}
\label{s:intro}

Let $G$ be a connected reductive group over an algebraically closed field $k$ of characteristic 0.

Let $k_0\subset k$ be a subfield such that $k$ is an algebraic closure of $k_0$.
Let $G_0$ be a {\em $k_0$-form} ($k_0$-model) of $G$, i.e. an algebraic group over $k_0$
together with an isomorphism of algebraic $k$-groups $\vk_G \colon G_0\times_{k_0} k\isoto G$.

Let $Y$ be an irreducible $G$-variety. By a {\em  $G_0$-equivariant $k_0$-form} of $Y$ we mean a $k_0$-variety $Y_0$, together with an isomorphism of $k$-schemes $\vk_Y \colon Y_0\times_{k_0}k \isoto Y$ and a morphism of $k_0$-schemes $G_0\times_{k_0}Y_0 \to Y_0$ defining an action of $G_0$ on $Y_0$, such that the following diagram commutes;
\[\xymatrixcolsep{3pc}\xymatrix{
G_{0,k}\times_k Y_{0,k}\ar[r]^-{\theta_{0,k}}\ar[d]_{\vk_G\times_k \vk_Y}&Y_{0,k}\ar[d]^{\vk_Y}\\
G\times_k Y\ar[r]^{\theta}&Y
}
\]
Here $\theta,\theta_{0,k}$ are the $G$-action on $Y$, and the base change to $k$ of the $G_0$-action on $Y_0$ respectively, and $G_{0,k},Y_{0,k}$ are the base-changes to $k$ of $G_0,Y_0$ respectively.

From now on till the end of the Introduction we assume that  $Y$ is a {\em spherical homogeneous space of $G$.}
This means that $Y=G/H$ (with the natural action of $G$) for some algebraic subgroup $H\subset G$
and that a Borel subgroup $B$ of $G$ has an open orbit in $Y$. See Timashev \cite{Tim} and Perrin \cite{Perrin} for surveys on the theory of spherical varieties.

Inspired by the works of Akhiezer and Cupit-Foutou  \cite{ACF},
for a given $k_0$-form $G_0$  of $G$, we ask whether $H$ is conjugate to a subgroup defined over $k_0$, i.e. whether there exists a $k_0$-subgroup $H_0\subset G_0$ such that $H_0 \times_{k_0}k\subset G$ is conjugate to $H$.

More generally, we ask whether there exists a $G_0$-equivariant $k_0$-form $Y_0$ of $Y$ (note that if $H$ is conjugate to $H_0\times_{k_0}k$, then $Y_0:=G_0/H_0$ is a $G_0$-equivariant $k_0$-form of $Y$).

A field $k_0$ is called \emph{large} if for any irreducible $k_0$-variety $Y_0$ having  a smooth $k_0$-point,
the $k_0$-points are Zariski-dense in $(Y_0)_k$; see Pop \cite[Proposition 2.6]{Pop}.
Any field that is complete with respect to a nontrivial absolute value is large; see \cite[Section 1, (A) (2)]{Pop}.
In particular, the field of real numbers $\R$ and any $p$-adic field (a finite extension of the field of $p$-adic numbers $\Q_p$) are large.

Let $T\subset B \subset G$ be a Borel subgroup of $G$ and a maximal torus contained in it. Let
\[\BRD(G)=\BRD(G,T,B)=(X,X^\vee, R, R^\vee,S, S^\vee)\]
denote the based root datum of $G$.
Here, in particular, $X=\X^*(T)=\X^*(B)$ is the character group of $T$ and $B$, $R$ is the root system,
and $S\subset R$ is the system of simple roots defined by $B$; see \Cref{ss:BRD} below for details.

Following Tits \cite[Section 2.3]{Tits}, we consider the $*$-action of $\Gamma$ on $\BRD(G)$ defined by the $k_0$-form $G_0$ of $G$;
see \Cref{ss:star-action} below.
We obtain a homomorphism
\[\ve\colon \Gamma\to\Aut\,\BRD(G).\]

The \emph{combinatorial invariants} of a spherical variety $G/H$ are the \emph{weight lattice} $\sX\subset \X^*(B)=\X^*(T)$,
the \emph{valuation cone}  $\sV\subset V:=\Hom(\sX,\Q)$,
and two finite subsets $\Omega^{(1)},\,\Omega^{(2)}\subset V\times\sP(S)$  related to the \emph{set of colors} of $G/H$.
Here $\sP(S)$ denotes the set of all subsets of $S$. See Borovoi \cite[Section 7]{BG}.

For $\gamma \in \Gamma$,  $\ega$ acts on $\X^*(B)$ and on $S$, and therefore,
it is clear what the phrase ``$\ega$ preserves the combinatorial invariants of $Y=G/H$'' means
(see Borovoi \cite[Section 8]{BG} for details).

The $k_0$-form $G_0$ is called \emph{quasi-split} if it contains a Borel subgroup defined over $k_0$.

The normalizer of a subgroup $H$ of $G$ is denoted by $\sN_G(H)$.

The following theorem is a generalization of Theorem 4.4 of Akhiezer and Cupit-Foutou \cite{ACF}, where the authors consider  the case
when $k_0=\R$ and $G_0$ is split.

\begin{theorem}[\Cref{t:main-qs}]\label{t:main-qs'}
Let $G$ be a reductive group over an algebraically closed field $k$.
Let $H\subset G$ be a spherical subgroup.
Let $k_0\subset k$ be a subfield such that $k$ is an algebraic closure of $k_0$.
Let $G_0$ be a $k_0$-form of $G$.
Assume that:
\begin{enumerate}
	\item[(i)]  $\Gamma:=\Gal(k/k_0)$ preserves the combinatorial invariants of $G/H$ when acting on $\BRD(G)$ via the homomorphism
$\ve\colon\Gamma\to\Aut\,\BRD(G)$  defined by the $k_0$-form $G_0$ of $G$;
    \item[(ii)] $\charr(k_0)=0$;
	\item[(iii)] $k_0$ is a large field;
	\item[(iv)] $\sN_G(\sN_G(H))=\sN_G(H)$;
    \item[(v)]  $G_0$ is quasi-split.

\end{enumerate}
Then $H$ is conjugate to a subgroup defined over $k_0$,
i.e. there exists a $k_0$-subgroup $H_0\subset G_0$ such that $(H_0)_k$ is conjugate to $H$.
\end{theorem}

The proof follows  ideas of Akhiezer and Cupit-Foutou \cite{ACF}.

\begin{remark}\label{c:form}
Under the assumptions of \Cref{t:main-qs'},  $Y=G/H$ admits a $G_0$-equivariant $k_0$-form $Y_0$.
 Indeed, we may take $Y_0:=G_0/H_0$.
 \end{remark}

\begin{remark}
The normalizer of a spherical subgroup in general is not self-normalizing; see Avdeev \cite[Example 4, Section 4]{Av13}.
In \cite{BG2} Borovoi and Gargliardi prove that conditions (iii) and (iv) in \Cref{t:main-qs'} are in fact unnecessary, but their proof is much more involved.
\end{remark}

 \begin{remark}\label{r:Ronan}
 In \cite[the proof of Corollary 2.7]{MJT2} Moser-Jauslin and Terpereau show that condition (iv) above is satisfied if $H\subseteq G$ is \emph{symmetric} (i.e., if there exists a nontrivial involution $\theta\in \Aut_k(G)$ such that $G^\theta \subseteq H \subseteq \sN_G(G^\theta)$). They state and prove this over $\mathbb{C}$, but the proof works over an arbitrary algebraically closed field of characteristic zero.

 They deduce, using \Cref{t:main-qs'} (which they restate and reprove), that if $H\subseteq G$ is a symmetric subgroup of a connected reductive complex group $G$, and $G_0$ is a quasi-split $\mathbb{R}$-form of $G$ (with corresponding anti-holomorphic involution $\sigma$), such that $\sigma(H)$ is conjugate to $H$, then there exists a $k$-subgroup $H'\subseteq G$ which is conjugate to $H$, such that $\sigma(H')=H'$. Again, they state and prove this over $\mathbb{R}$, but the proof generalizes easily to an arbitrary large field of characteristic zero.
 \end{remark}

See \Cref{s:ssubgroup} for some examples where \Cref{t:main-qs'} applies.

Note that assumption (i) is a necessary condition for $G/H$ to have a $k_0$-form; see Borovoi \cite[Proposition 8.12]{BG}.

If assumptions (i)-(iv) of \Cref{t:main-qs'} are satisfied, but $G_0$ is not assumed to be quasi-split,
we can give a necessary and sufficient Galois-cohomological condition
for the existence of a $G_0$-equivariant $k_0$-form of $G/H$.

\begin{theorem}[\Cref{t:main-general}]\label{t:main-general'}
Let $G$ be a reductive group over an algebraically closed field $k$.
Let $H\subset G$ be a spherical subgroup.
Let $k_0\subset k$ be a subfield such that $k$ is an algebraic closure of $k_0$.
Let $G_0$ be a $k_0$-form of $G$.
Assume that:
\begin{enumerate}
	\item[(i)]  $\Gamma:=\Gal(k/k_0)$ preserves the combinatorial invariants of $G/H$ when acting on $\BRD(G)$ via the homomorphism
$\ve\colon\Gamma\to\Aut\,\BRD(G)$  defined by the $k_0$-form $G_0$ of $G$;
    \item[(ii)] $\charr(k_0)=0$;
	\item[(iii)] $k_0$ is a large field;
	\item[(iv)] $\sN_G(\sN_G(H))=\sN_G(H)$.
\end{enumerate}
Then $Y=G/H$ admits a $G_0$-equivariant $k_0$-form if and only if
the Tits class $t(\Gtil_0)\in H^2(k_0,Z(\Gtil_0))$ has trivial image in $H^2(k_0,A_\qs)$ (see \Cref{ss:Tits} ).
\end{theorem}

We explain the notation.
Let $\Gtil_0$ denote the universal covering of the commutator subgroup $[G_0,G_0]$ of $G_0$.
We write $\Ztil_0$ for the center of the simply connected group $\Gtil_0$, which is a finite abelian $k_0$-group.
Write $A=\sN_G(H)/H$, which is an abelian group of multiplicative type (i.e., it is a closed subgroup of a torus);
see Losev \cite[Theorem 2 and Definition 4.1.1(1)]{Losev}.
Then we have a canonical isomorphism $A(k)\isoto\Aut^G(Y)$; see e.g. Borovoi \cite[Lemma 5.1]{BG}.
We have a canonical homomorphism of abelian $k$-groups
\[ \Ztil\to Z(G)\into\sN_G(H)\to A,\]
where $\Ztil$ is the center of  $\Gtil$.
Using the $*$-action $\ve\colon\Gamma\to\Aut\,\BRD(G,T,B)$, we define a $k_0$-form $A_\qs$ of $A$ and a $k_0$-homomorphism
\[\vktil\colon \Ztil_0\to A_\qs\hs.\]
See \Cref{ss:Aqs} below.
This homomorphism induces a homomorphism on the second cohomology
\[\vktil_*\colon H^2(k_0,\Ztil_0)\to H^2(k_0,A_\qs).\]
The $k_0$-form $\Ztil_0$ of $\Ztil$, the $k_0$-form $A_{qs}$ of $A$, the $k_0$-homomorphism $\vktil:\Ztil_0\to A_{qs}$, and the induced homomorphism on cohomology $\vktil_*$, can be constructed in terms of the homomorphism $\ve$; see Borovoi and Gagliardi \cite[Section 3]{BG2}.
Now, the $k_0$-form $\Gtil_0$ of $\Gtil$ defines the \emph{Tits class} $t(\Gtil_0)\in H^2(k_0,\Ztil_0)$; see \Cref{ss:Tits} below.
\Cref{t:main-general'} says that $G/H$ admits a $G_0$-equivariant $k_0$-form
if and only if the class $\vktil_*(t(\Gtil_0))\in H^2(k_0,A_\qs)$ is  trivial.

Note that if $G_0$ is quasi-split, then $t(\Gtil_0)=1$ and hence $\vktil_*(t(\Gtil_0))=1$.
We see that \Cref{t:main-general'} generalizes the assertion of \Cref{c:form}.
Our proof of \Cref{t:main-general'} uses  \Cref{t:main-qs'}, see \Cref{t:main-general} below.

For examples where \Cref{t:main-general'} applies, see Borovoi and Gagliardi \cite{BG2}.

The plan for the rest of the paper is as follows. In \Cref{s:pre} we recall basic definitions and results related to forms and semilinear actions, and the $*$-action on the based root datum of a reductive group induced from a $k_0$-form.
In \Cref{s:spherical vars} we recall basic definitions and facts related to the combinatorial invariants of spherical varieties.
In \Cref{s:existence} we show that if $\sN_G(H)=H$ and $\ega$ preserves the combinatorial invariants of $G/H$ then the (unique) wonderful embedding $\iota \colon G/H \into W$ admits a $G_0$-equivariant $k_0$-form.
In \Cref{s:rational} we show that if $W_0$ is a $G_0$-equivariant $k_0$-form of a wonderful $G$-variety $W$,
and the field $k_0$ is large, then the open orbit $G/H\subset W$ admits a $k_0$-rational point.
In \Cref{s:ssubgroup} we combine these results to prove \Cref{t:main-qs'}, and give some examples.
 In \Cref{s:final} we deduce \Cref{t:main-general'} from \Cref{t:main-qs'} and Borovoi and Gagliardi \cite[Theorem 1.6]{BG2}.

\begin{remark}
Since the original writing of this text, this work was used by Moser-Jauslin and Terpereau in \cite{MJT2}, and generalized by Borovoi and Gagliardi in \cite{BG2}.
\end{remark}

{\textsc {Acknowledgements.}}
The author is grateful to Mikhail Borovoi and Giuliano Gagliardi for stimulating discussions.

\noindent
{\bf Notation and assumptions.}\\
$k$ is an algebraically closed field.\\
$k_0$ is a subfield of $k$ such that $k$ is a Galois extension of $k_0$, hence $k_0$ is perfect. We denote the Galois group $\Gal(k/k_0)$ by $\Gamma$.\\
A {\em $k$-variety} is a reduced separated scheme of finite type over $k$, not necessarily irreducible.\\
An {\em algebraic $k$-group} is a smooth $k$-group scheme of finite type over $k$, not necessarily connected.
All {algebraic $k$-subgroups} are assumed to be smooth.\\
All homogeneous spaces are assumed to be spherical.\\
$G$ is a connected reductive algebraic $k$-group and $G_0$ is a $k_0$-form of $G$.\\
$H\subset G$ is a spherical $k$-subgroup.\\
For a field $k$, we write $\times_k$ for $\times_{\Spec k}$. For a field extension $k_1/k_0$, we write $\times_{k_0} k_1$ for $\times_{\Spec k_0} \Spec k_1$.

\section{Preliminaries}
\label{s:pre}
\begin{subsec}\label{ss:semilinear}
Let $k_0$ be a field of characteristic zero with a fixed algebraic closure $k$.
By a $k_0$-form of a $k$-scheme $X$ we mean a $k_0$-scheme $X_0$ together with an isomorphism of $k$-schemes
\[\vk_X \colon X_0\times_{k_0}k \isoto X.\]
Two $k_0$-forms $(X_0,\vk_X),(X_0^{'},\vk_{X'})$ of $X$ are said to be isomorphic if there exists an isomorphism of $k_0$-schemes $f\colon X_0\isoto X_0^{'}$.

We write $\Gamma=\Gal(k/k_0)$. For $\gamma \in \Gamma$, denote by $\gamma^{*} \colon \Spec k \to \Spec k$ the morphism of schemes induced by $\gamma$. Notice that $(\gamma \circ  \delta)^{*}=\delta^{*} \circ \gamma^{*}$.\\
Let $(X,p_X \colon X\to \Spec k)$ be a $k$-scheme. For every $\gamma \in \Gamma$ we define the $\gamma$-conjugated scheme $(^{\gamma}X,^{\gamma}p_X)$ to be the $k$-scheme which is the same abstract scheme as $X$, but with the structural morphism $^{\gamma}p_X\colon=(\gamma^{*})^{-1}\circ p_X\colon X \to \Spec k$. It is an elementary fact that $(^{\gamma}X,^{\gamma}p_X)$ realizes the pull-back of $X\xrightarrow{p_X} \Spec k$ along $\gamma^{*}\colon \Spec k\to \Spec k$.

A $k/k_0$\,-semilinear automorphism of $X$ is a pair $(\gamma,\mu)$ where $\gamma\in \Gamma$ and $\mu\colon X\to X$
is an isomorphism of schemes such that the diagram below commutes:
\begin{equation*}
\xymatrix@R=25pt@C=40pt{
X\ar[r]^\mu \ar[d]_{p_X}          & X\ar[d]^{p_X} \\
\Spec k\ar[r]^-{(\gamma^*)^{-1}}   & \Spec k
}
\end{equation*}
In this case we say also that $\mu$ is a $\gamma$-semilinear automorphism of $X$.
We shorten ``$\gamma$-semilinear automorphism'' to ``$\gamma$-semi-automorphism''.
Note that if $(\gamma,\mu)$ is a semi-automorphism of $X$, then $\mu$ uniquely determines $\gamma$;
see Borovoi \cite[Lemma 1.2]{BG}.

We denote $\SAut(X)$ the group of all $\gamma$-semilinear automorphisms $\mu$ of $X$ where $\gamma$ runs over $\Gamma=\Gal(k/k_0)$.
There is a cannonical bijection (which is set-theoretically the identity) between $\gamma$-semilinear automorphisms of $X$, and $k$-isomorphisms $^{\gamma}X\isoto X$.
By a semilinear action of $\Gamma$ on $X$ we mean a homomorphism of groups
\[\mu \colon \Gamma \rightarrow \SAut(X), \quad \gamma \mapsto \mu_\gamma\]
such that for each $\gamma \in \Gamma$, $\mu_\gamma$ is $\gamma$-semilinear.

If we have a $k_0$-scheme $X_0$, then the formula $\gamma \mapsto \id_{X_0} \times (\gamma^{*})^{-1}$  defines a semilinear action of $\Gamma$ on $X:=X_0\times_{k_0}k$. Thus a $k_0$-form of $X$ induces a semilinear action of $\Gamma$ on $X$. If $X$ is of finite type then two forms $X_0,X_0'$ are isomorphic if and only if they induce the same semilinear action $\Gamma \to \SAut(X)$, up to conjugation by an element of $\Aut(X)$.

If $(G,p_G \colon G \to \Spec k)$ is an algebraic $k$-group, we define a $k_0$-form of $G$ similarly as above, with the additional assumption that the isomorphism $\vk_G$ is an isomorphism of algebraic $k$-groups (rather than just $k$-schemes). An isomorphism of $k_0$-forms of $G$ is defined as above.

For $\gamma \in \Gamma$, the $\gamma$-conjugated scheme $^{\gamma}G$ has a canonical structure of an algebraic $k$-group (since it is the pullback of an algebraic $k$-group along the morphism $\gamma^{*}\colon \Spec k\to \Spec k$).

A $\gamma$-semi-linear automorphism of $G$ is defined as above, with the additional assumption that the $k$-morphism $\mu \colon ^{\gamma}G\to G$ is a morphism of algebraic $k$-groups.

We denote $\SAut(G)$ the group of all $\gamma$-semilinear automorphisms $\tau$ of $G$ where $\gamma$ runs over $\Gamma=\Gal(k/k_0)$.
By a semilinear action of $\Gamma$ on $G$ we mean a homomorphism
\[\sigma \colon \Gamma \to \SAut(G), \quad \gamma \mapsto \siga\]
such that for all $\gamma \in \Gamma$, $\siga$ is $\gamma$-semilinear.
As above, a $k_0$-form $G_0$ of $G$ induces a semilinear action of $\Gamma$ on $G$. Two $k_0$-forms of $G$ are isomorphic if and only if they induce the same semilinear action of $\Gamma$ on $G$, up to conjugation by an element of $\Aut(G)$.

Let $G$ be an algebraic group over $k$ and let $X$ be a $G$-$k$-variety.
Let $G_0$ be a $k_0$-form of $G$. It gives rise to a semilinear action $\sigma \colon \Gamma \rightarrow \SAut(G),\gamma \mapsto \siga$. Let $X_0$ be a $G_0$-equivariant $k_0$-form of $X$, i.e. a $k_0$-form $X_0$ of $X$ together with a morphism $G_0 \times_{k_0} X_0 \to X_0$ compatible with the action $G\times_k X\to X$.
It gives rise to a semilinear action $\mu \colon \Gamma \to \SAut(X)$ such that for all $\gamma$ in $\Gamma$:
\[\mu_\gamma(g \cdot x)=\sigma_\gamma(g) \cdot \mu_\gamma(x) \ \forall x\in X(k),g\in G(k).\]
We say then that $\muga$ is $\siga$-equivariant.

\end{subsec}

\begin{subsec}\label{ss:inner}
Let $k$ be an algebraically closed field, and
let $G$ be a linear algebraic group over $k$.
We denote by $\Aut(G)$ the group of $k$-automorphisms of $G$, regarded as an abstract group.
For $g\in G(k)$, we denote conjugation by $g$ by $i_g\in\Aut(G)$, and we denote
$\Inn(G)=\{i_g\ |\ g\in G(k)\}$.

Let $k_0\subset k$ be a subfield such that $k$ is a Galois extension of $k$.
We write $\Gamma=\Gal(k/k_0)$.
Let $G_0$ be a $k_0$-form of $G$, it defines a semilinear action
\[ \sigma\colon \Gamma\to \SAut(G).\]
This action induces an action of $\Gamma$ on the abstract group $\Aut(G)$ by
\[(^\gamma a)(g)=\sigma_\gamma(a(\sigma_\gamma^{-1}(g)))\quad\text{for }g\in G(k),\ a\in\Aut(G).\]
Recall that a map
\[c\colon \Gamma\to \Aut(G)\]
is called a 1-\emph{cocycle} if the map $c$ is continuous (for the profinite topology on $\Gamma$ and the discrete topology on $\Aut(G)$) and satisfies the following condition:
\begin{equation}\label{e:cocycle}
c_{\gamma\delta}=c_\gamma\cdot \upgam c_\delta\text{ for all }\gamma,\delta\in\Gamma.
\end{equation}
The set of such cocycles is denoted by $Z^1(\Gamma,\Aut(G)\hs)$ or $Z^1(k_0,\Aut(G)\hs)$.
For $c\in Z^1(k_0,\Aut(G)\hs)$, we consider the $c$-twisted semilinear action
\[\sigma'\colon \Gamma\to \SAut(G),\quad \gamma\mapsto c_\gamma\circ \sigma_\gamma.\]
Then, clearly, $\sigma'_\gamma$ is a $\gamma$-semi-automorphism of $G$ for any $\gamma\in\Gamma$.
It follows from the cocycle condition \eqref{e:cocycle} that
\[\sigma'_{\gamma\delta}=\sigma'_\gamma\circ\sigma'_\delta\ \text{ for all }\gamma,\delta\in\Gamma.\]
Since $G$ is an affine variety, the semilinear action $\sigma'$ comes from some $k_0$-form $G_0'$ of $G$,
see Serre \cite[III.1.3, Proposition 5]{Serre}.
We write $G_0'=\hs_c(G_0)$ and say that $G_0'$ is the \emph{twisted form of $G_0$ defined by the cocycle} $c$.

If $c\in Z^1(\Gamma,\Inn(G))$ is a 1-cocycle with coefficients in $\Inn(G)$, then we say that $_c G_0$ is an \emph{inner twist} or \emph{inner form} of $G_0$.

\end{subsec}

\begin{subsec}\label{ss:BRD}
From now on we assume that $G$ is a connected reductive group over an algebraically closed field $k$  of characteristic 0
and that $H\subset G$ is a spherical $k$-subgroup.
Let $B\subset G$ be a Borel subgroup, and let $T\subset B$ be a maximal torus.
We consider the based root datum
\[ \BRD(G)=\BRD(G,T,B)=(X,X^\vee, R,R^\vee, S, S^\vee),\]
where\\
 $X=\X^*(T):=\Hom(T,\G_{m,k})$ is the character group of $T$;\\
 $X^\vee=\X_*(T):=\Hom(\G_{m,k}, T)$ is the cocharacter group of $T$;\\
 $R=R(G,T)\subset X$ is the root system;\\
 $R^\vee\subset X^\vee$ is the coroot system;\\
 $S=S(G,T,B)\subset R$ is the system of simple roots (the basis of $R$) defined by $B$;\\
 $S^\vee\subset R^\vee$ is the system of simple coroots.

 There is a canonical pairing $X\times X^\vee\to \Z,\ (\chi,x)\mapsto \langle \chi,x\rangle$,
 and a canonical bijection $\alpha\mapsto \alpha^\vee\colon R\to R^\vee$
 such that $S^\vee=\{\alpha^\vee\ |\ \alpha\in S\}$.
See Springer \cite[Sections 1 and 2]{Springer} for details.
\end{subsec}

\begin{subsec} \label{ss:star-action}
We write $\SAut_{k/k_0}(G)$ or just $\SAut(G)$ for the group of $k/k_0$\,-semi-automorphism of $G$.
Let $T$ and $B$ be as in \Cref{ss:BRD}.
We construct a canonical homomorphism (cf. Tits \cite[Section 2.3]{Tits})
\begin{equation}\label{e:phi-SAut}
\phi\colon \SAut(G)\to\Aut\,\BRD(G,T,B).
\end{equation}
Let $(\gamma,\tau)$ be a semi-automorphism of $G$.
Consider $\tau(T)\subset \tau(B)\subset G$.
Then $\tau(T)$ is a maximal torus and $\tau(B)$ is a Borel subgroup.
By  Theorem 11.1 and Theorem 10.6(4) in Borel's book \cite{Borel},
there exists an element $g\in G(k)$ such that
\[ g\cdot\tau(T)\cdot g^{-1}=T\quad \text{and}\quad g\cdot \tau(B)\cdot g^{-1}=B.\]
We obtain a $\gamma$-semi-automorphism
$\tau':=\inn(g)\circ\tau$.
The pair $(\gamma,\tau')$ induces an automorphism of the based root datum
\[\phi(\gamma,\tau)\in\Aut\,\BRD(G,T,B)\]
as follows:
Since $\tau'(T)=T$, the pair $(\gamma,\tau')$ acts on the character group $X=\X^*(T)$ by $\chi\mapsto(\gamma,\tau')(\chi)$, where
$$(\gamma,\tau')(\chi)(t)= \hs^\gamma(\chi((\tau')^{-1}(t))),$$
and similarly it acts on the dual lattice $X^\vee=\X_*(T)$ and preserves the pairing $X\times X^\vee\to\Z$.
One can easily check that $(\gamma,\tau')$, when acting on $\Lie(G)$, takes the eigenspace corresponding to a root $\alpha\in R$
to the eigenspace corresponding to $(\gamma,\tau')(\alpha)$, hence $(\gamma,\tau')$ preserves $R$, and, similarly,
it preserves, $R^\vee$ and the bijection $\alpha\mapsto\alpha^\vee\colon R\to R^\vee$.
Since $(\gamma,\tau')$ preserves $B$, it preserves the set of positive roots with respect to $B$, hence, it preserves $S$ and $S^\vee$.
If $g'\in G(k)$ is another element such that
\[ g'\cdot\tau(T)\cdot(g')^{-1}=T\quad \text{and}\quad g'\cdot\tau(B)\cdot (g')^{-1}=B,\]
then $g'=tg$ for some $t\in T(k)$.
Since $\inn(t)$ acts trivially on $\BRD(G,T,B)$,
we see that $\phi(\gamma,\tau)$ does not depend on the choice of $g$.
One checks easily that $\phi\colon \SAut(G)\to\Aut\,\BRD(G,T,B)$ is a homomorphism; see \cite[Proposition 3.1]{BKLR}.

If $G_0$ is a $k_0$-form of $G$ with corresponding semilinear action $\sigma \colon \Gamma \to \SAut(G)$, we denote by
\[\ve = \phi \circ \sigma \colon \Gamma \to \Aut\, \BRD(G,T,B), \quad \gamma \mapsto \ega \]
the induced action of $\Gamma$ on $\BRD(G,T,B)$, and call it ``the $\ast$-action of $\Gamma$ on $\BRD(G,T,B)$ induced by the $k_0$-form $G_0$'' (or simply ``the $\ast$-action'' when there is no risk of confusion).
\end{subsec}

Now let $G$ be a connected reductive group, and let $G_0$ be a $k_0$-form of $G$ defining a semilinear action $\sigma\colon \Gamma\to\SAut(G)$.
Then a subgroup $H\subset G$ is ``defined over $k_0$'' if and only if it is $\Gamma$-invariant.
Recall that a $k_0$-form $G_0$ of $G$ is called \emph{split} if $G_0$ has a split maximal $k_0$-torus.
All split $k_0$-forms of $G$ are isomorphic.
A $k_0$-form $G_0$ of $G$ is called \emph{quasi-split} if $G$ has a Borel subgroup defined over $k_0$.
Every split $k_0$-form of $G$ is quasi-split.

We shall need the following well-known result:

\begin{proposition}\label{p:inn-qs}
Let $G_0$ be a connected reductive group over a field $k_0$.
Then there exist a unique (up to isomorphism)  quasi-split inner form $G_\qs$ of $G_0$. Moreover, $G_0$ and $G_\qs$ induce the same action of $\Gamma$ on $\BRD(G,B,T)$.
\end{proposition}

\begin{proof}
See The Book of Involutions \cite[Proposition(31.5)]{KMRT}, or Conrad \cite[Proposition 7.2.12]{Conrad},
 the existence only in  Springer \cite[Proposition 16.4.9]{Springer2}.
\end{proof}

\section{Spherical varieties and their combinatorial invariants}\label{s:spherical vars}
From now on all varieties are assumed to be geometrically integral.

A normal $G$-variety $X$ is called \emph{spherical} if $B$ has an open orbit in $X$ for some (equivalently: for every) Borel subgroup $B\subseteq G$. This is equivalent to there being only finitely many $B$-orbits in $X$ for some (equivalently: any) Borel subgroup $B\subseteq G$ (see \Cref{t:emb} below for the nontrivial implication).

$X$ is called a homogeneous spherical variety if $X$ is spherical and $X \cong G/H$ as $G$-varieties for some closed subgroup $H \subseteq G$.

If $X$ is a spherical $G$-variety, then it has an open $G$ orbit, and hence, admits an open embedding $\iota \colon G/H \into X$ where $G/H$ is a homogeneous spherical $G$-variety.

Thus the classification of spherical varieties splits into two parts: \begin{enumerate}
\item[(i)] Given a homogeneous spherical $G$-variety $Y=G/H$, classify all possible open $G$-embeddings $Y\into X$ of $Y$ into a normal $G$-variety $X$.
\item[(ii)] Classify all homogeneous spherical $G$-varieties.
\end{enumerate}

\subsec{Homogeneous spherical varieties and their combinatorial invariants}\label{ss:s-homog}

Let $Y=G/H$ be a spherical homogeneous $G$-variety. We associate to $Y$ the following combinatorial data:

Let $K(Y)$ denote the field of rational functions on $Y$ (notice that $G(k)$ naturally acts on $K(Y)$). For any character $\chi \in X^*(B)$ we define
\[K(Y)^{(B)}_{\chi}=\{f\in K(Y)\  |\  b\cdot f = \chi(b) \cdot f \text{ for all } b\in B(k)\}\]
Because $Y$ is spherical, $\dim_k K(Y)^{(B)}_{\chi}\leq 1$ for all $ \chi \in X^*(B)$. We define the \emph{weight lattice} $\sX$ of $G/H$ by
\[\sX=\sX(Y)=\{\chi \in X^*(B)\: | \:\dim_k K(Y)^{(B)}_{\chi} \neq 0\}\]
This is called the weight lattice of $G/H$. Clearly, $\sX$ is a subgroup of $X^*(B)$. We define $V=V(Y) = \Hom_\Z (\sX, \Q)$.

Let $\Val(K(Y))$ denote the set of $\Q$-valued valuations on $K(Y)$ that are trivial on $k$. We have a natural action of $G(k)$ on $\Val (K(Y))$, and we denote by $\Val^G(K(Y))$ (resp. $\Val^B(K(Y))$) the set of $G$-invariant (resp. $B$-invariant) valuations.

We have a natural map $\rho \colon \Val^B(K(Y)) \rightarrow V$ given by $ v \mapsto (\chi \mapsto v(f_\chi))$. It is known (see Knop \cite[Corollary 1.8]{Knop-LV}) that the restriction of $\rho$ to $\Val^G(K(Y))$ is injective. Denote the image
\[\sV=\sV(Y)=\rho (\Val^G(K(Y)))\subseteq V\]
We call $\sV$ the \emph{valuation cone} of $Y$.

Let $\sD$ denote the set of $B$-invariant prime divisors in $Y$ (these are called ``colors''). This is a finite set (it may be empty). Each element $D \in \sD$ defines a $B$-invariant valuation denoted $\val(D)\in \Val^B(K(Y))$ given by sending a rational function $f\in K(Y)$ to its order of vanishing at $D$. By abuse of notation we write $\rho$ for $\rho \circ \val \colon \sD \rightarrow V$. In general, this map is not injective.

For every $D \in \sD$ denote by $\Stab_G(D)$ the stabilizer of $D\subset Y$ in $G$. Since $D$ is $B$-invariant, we obviously have $B\subseteq \Stab_G(D)$. For every $\alpha \in S$ we denote by $P_\alpha$ the corresponding minimal parabolic subgroup in $G$ containing $B$. Denote by $\sP(S)$ the power set of $S$. Define $\vs \colon \sD \rightarrow \sP (S)$ by letting $\vs (D)$ be the set of all $\alpha \in S$ such that $P_\alpha \not\subseteq \Stab_G(D)$.

\begin{lemma}
Every fiber of the map $\vs \colon \sD \rightarrow \sP (S)$ consists of \ $\leq 2$ elements.
\end{lemma}
\begin{proof}
See \cite[Lemma 7.1]{BG}
\end{proof}

\begin{corollary}
Every fiber of the map $\rho \times \vs \colon \sD \rightarrow V \times \sP (S)$ consists of \ $\leq 2$ elements.
\end{corollary}

Denote by $\Omega^{(1)}$ (resp. $\Omega^{(2)}$) the subset of $V \times \sP(S)$ consisting of those elements with precisely one (resp. two) preimages under the map $\rho \times \vs$.

\begin{definition}
Let $G$ be a connected reductive group over $k$, $Y$ a spherical homogeneous space for $G$. By the \emph{combinatorial invariants} of $Y$ we will mean:$$\sX\subseteq X^*(B),\quad \sV \subseteq V, \text{ and }  \Omega^{(1)},\Omega^{(2)} \subseteq V \times \sP(S)$$
\end{definition}

The combinatorial invariants of a spherical homogeneous space completely determine it as a $G$-variety:

\begin{theorem}[Losev's Uniqueness Theorem {\cite[Theorem 1]{Losev}}]
Let $G$ be a connected reductive group over $k$, $H_1,H_2$ two spherical subgroups, $Y_1,Y_2$ the corresponding spherical homogeneous spaces. Assume that $\charr k = 0$. Then $H_1$ and $H_2$ are conjugate in $G$ if and only if $$\sX(Y_1)=\sX(Y_2),\quad \sV(Y_1)=\sV(Y_2), \text { and } \Omega^{(1)}(Y_1)=\Omega^{(1)}(Y_2),\  \Omega^{(2)}(Y_1)=\Omega^{(2)}(Y_2)$$
i.e., if and only if they have the same combinatorial invariants.
\end{theorem}

Now let $H\subseteq G$ be a spherical subgroup, $G_0$ a $k_0$-form of $G$ with corresponding semilinear action $\sigma \colon \Gamma \rightarrow \SAut(G)$, and $\gamma \in \Gamma$. Denote $Y=G/H$ and $\ega \in \Aut (X^*(B),S)$ the automorphism arising from $\siga$.
\begin{definition}\label{d:preserves}
We say that $\ega$ preserves the combinatorial invariants of $Y$ if
\[
\ega (\sX(Y))=\sX(Y),\quad \ega (\sV (Y))=\sV (Y), \text{ and } \ega(\Omega^{(i)}(Y))=\Omega^{(i)}(Y), i=1,2
\]
Here, once we require $\ega (\sX(Y))=\sX(Y)$, we get induced automorphisms of $V(Y)=\Hom_\Z(\sX(Y),\Q)$ and $V(Y)\times \sP(S)$, which (abusing notations) we also denote by $\ega$.
\end{definition}

\begin{remark}
This condition is automatic when $G_0$ is split, since in this case $\ega$ acts trivially on $(X^*(B),S)$.
\end{remark}

\subsec{Spherical embeddings}

Let $Y=G/H$ be a homogeneous spherical $G$-variety.

A \emph{colored cone} is a pair $(\sC,\sF)$ where $\sF \subseteq \sD$ is a subset and $\sC \subseteq V$ is a strictly convex cone generated by $\rho(\sF)$ and finitely many elements of $\sV$, such that $0\notin \rho(\sF)$.

A face of a colored cone $(\sC,\sF)$ is a colored cone $(\sC_0,\sF_0)$ such that $\sC_0$ is a face of $\sC$ and $\sF_0=\sF\cap \rho^{-1}(\sC_0)$.

A colored fan is a nonempty finite set $\Sigma$ of colored cones satisfying:
\begin{enumerate}
     \item[(i)] For every $(\sC,\sF)\in \Sigma$, $\sC^\circ \cap \sV \neq \emptyset$ (where $\sC^\circ$ denotes the relative interior of $\sC$).
     \item[(ii)] For every face $(\sC_0,\sF_0)$ of a colored cone $(\sC,\sF)\in \Sigma$ satisfying $\sC_0^\circ \cap \sV \neq \emptyset$, $(\sC_0,\sF_0)$ also belongs to $\Sigma$.
     \item[(iii)] For every $u\in \sV$ there is at most one $(\sC,\sF)\in \Sigma$ with $u\in \sC^\circ$.
\end{enumerate}

To any spherical embedding $G/H \into X$ we can associate a colored fan in the following manner:

Denote by $\sA$ the set of $G$-invariant prime divisors in $X$. Then $\Delta = \sD \cup \sA$ is the set of $B$-invariant prime divisors in $X$. The definition of $\rho \colon \sD \to V$ extends to $\rho \colon \Delta \to V$.

For any $G$-orbit $Y\subseteq X$ we denote by $I_Y\subseteq \Delta$ the set of $B$-invariant prime divisors containing $Y$, and we set
\[
(\sC_Y,\sF_Y):=(\text{cone}(\rho(I_Y)),I_Y\cap \sD).
\]

\begin{theorem}[Knop {\cite[Theorem 3.3 and the paragraph before it]{Knop-LV}}\hs]\label{t:emb}
The map
\[
(G/H\into X)\mapsto \Sigma (X):=\{(\sC_Y,\sF_Y):Y\subseteq X \text{ is a }G\text{-orbit}\}
\]
defines a bijection between isomorphism classes of spherical embeddings of $G/H$ and colored fans. Furthermore, the assignment $Y\mapsto (\sC_Y,\sF_Y)$ defines a bijection \[\{G\text{-orbits in }X\}\to \Sigma(X)\] such that for two $G$-orbits $Y,Z\subseteq X$ we have $Y\subseteq \overline{Z}$ if and only if $(\sC_Z,\sF_Z)$ is a face of $(\sC_Y,\sF_Y)$.
\end{theorem}

\subsec{Wonderful varieties}

The following definition is due to Luna \cite{Luna}.

\begin{definition}\label{d:wonderful}
A $G$-variety $W$ is called \emph{wonderful} if:
\begin{enumerate}
\item[(i)] $W$ is complete and smooth;
\item[(ii)] $W$ admits an open $G$-orbit whose complement consists of a finite union of smooth prime divisors $X_1,...,X_r$ with normal crossings.
\item[(iii)] The $G$-orbit closures of $W$ are given by partial intersections of the $X_i$'s, and they are all nonempty.
\end{enumerate}
\end{definition}

In particular, a wonderful $G$-variety $W$ has a unique closed $G$-orbit, which we denote by $X:=\bigcap_{i=1}^r X_i$.

Every wonderful $G$-variety is spherical; see Luna \cite{Luna}.

As Luna \cite{Luna} noticed, a wonderful variety is always projective (see also Avdeev \cite[Proposition 3.18]{Avdeev}).

A \emph{wonderful embedding} of a spherical homogeneous space $G/H$ is  an open $G$-embedding $\iota \colon G/H \into W$ into a wonderful $G$-variety.

Any spherical variety $G/H$ with $\sN_G(H)=H$  admits a wonderful embedding; see Knop \cite[Corollary 7.2]{Knop}.

Wonderful varieties have a very nice description in terms of colored fans:

\begin{theorem}\label{t:wonderfan}
Under the correspondence between colored fans and spherical embeddings, a wonderful embedding $G/H\into W$ corresponds to the colored fan that has the following two properties:
\begin{enumerate}
     \item[(i)] It has exactly one maximal colored cone, and it equals the valuation cone $\sV(Y)$.
     \item[(ii)] It has no colours.
\end{enumerate}
\end{theorem}
\begin{proof}
See Perrin \cite[Corollary 3.4.2, Definition 3.4.8 and Theorem 3.4.16]{Perrin}.
\end{proof}

\subsec{Galois descent for spherical embeddings}

Let $k_0\subseteq k$ be a Galois extension of fields such that $k$ is algebraically closed and of characteristic zero, and let $\Gamma =\Gal(k/k_0)$ the Galois group. Let $G$ be a connected reductive $k$-group, $G_0$ a $k_0$-model of $G$. Let $Y=G/H$ a spherical homogeneous $G$-variety, $Y_0$ a $G_0$-equivariant $k_0$-model of $Y$ with induced semilinear action
\[\Gamma\to\SAut(Y),\: \gamma \mapsto \sigma_{\gamma}^{(Y)}\]

\begin{theorem}
The $G_0$-equivariant $k_0$-form $Y_0$ of $Y$ naturally induces compatible actions of $\Gamma$ on
\[
K(Y),\Val^B(K(Y)),\sX,V,\sV,\sD.
\]
\end{theorem}
\begin{proof}
See Huruguen \cite[Subsection 2.2]{Huruguen}. (Huruguen assumes that $Y_0$ admits a $k_0$-point, but he does not use this assumption.)
\end{proof}

We call a fan $\Sigma = \{(\sC_i,\sF_i)\}_{i\in I}$ $\Gamma$-stable if for every $(\sC_i,\sF_i)\in \Sigma$ and for every $\gamma \in \Gamma$, the colored cone $(\gamma(\sC_i),\gamma(\sF_i))$ still belongs to $\Sigma$.

Let $\iota \colon Y\into X$ be a spherical embedding. We say that ``the semilinear action of $\Gamma$ extends to $X$'' if there exists a semilinear action
\[\Gamma \to \SAut(X),\:\gamma \mapsto \sigma_{\gamma}^{(X)}\]
such that for every $\gamma \in \Gamma$ one has \[\iota \circ \sigma_{\gamma}^{(Y)}=\sigma_{\gamma}^{(X)}\circ\iota\]

We will need the following result.
\begin{theorem}[Huruguen {\cite[Theorem 2.23]{Huruguen}}\hs]\label{t:Huruguen}
Let $k,k_0,G,G_0,Y,Y_0,X$ be as above. Then the semilinear action of $\Gamma$ extends to $X$ if and only if the colored fan $\Sigma(X)$ corresponding to $X$ is $\Gamma$-stable. Furthermore, in this case, for every $G$-orbit $Y$ in $X$ and every $\gamma \in \Gamma$ one has \[\gamma(\sC_Y)=\sC_{\gamma(Y)},\quad \gamma(\sF_Y)=\sF_{\gamma(Y)}.\]
\end{theorem}

\section{Equivariant forms of wonderful varieties}
\label{s:existence}

\begin{subsec}\label{ss:setup}
The setup for the rest of the paper is as follows: $k_0$ is a field of characteristic zero with a fixed algebraic closure $k$, and $\Gamma=\Gal(k/k_0)$ is the Galois group. We fix a connected reductive $k$-group $G$, a spherical $k$-subgroup $H\subseteq G$, and a $k_0$-form $G_0$ of $G$.
\end{subsec}

The following \Cref{l:Bo} and \Cref{t:Bo} are generalizations
of results of Akhiezer \cite{Akhiezer}, cf. Borovoi \cite{BG}.

\begin{lemma}[{\cite[Proposition 8.12]{BG}}\hs]
\label{l:Bo}
Let $k_0,k,\Gamma,G,H,G_0$ be as in \Cref{ss:setup}, and let $\gamma \in \Gamma$. Then the following are equivalent:\begin{enumerate}
    \item [(i)] $\ve_\gamma$ preserves the combinatorial invariants of $Y=G/H$;
    \item[(ii)] there exists $a_\gamma \in G(k)$ such that $\sigma_\gamma (H) = a_\gamma H {a_\gamma}^{-1}$
\end{enumerate}
\end{lemma}

\begin{theorem}[{\cite[Theorem 9.3]{BG}}\hs]
\label{t:Bo}
Let $k,k_0,\Gamma,G,H,G_0$ be as in \Cref{ss:setup}, and let $Y=G/H$. Assume that:
\begin{enumerate}
\item[(i)] For all $\gamma \in \Gamma$, $\ve_\gamma$ preserves the combinatorial invariants of $Y$;
\item[(ii)] $\sN_G(H)=H$;
\end{enumerate}
Then $Y$ admits a $G_0$-equivariant $k_0$-form $Y_0$. This $k_0$-form is unique up to a unique isomorphism.
\end{theorem}

We will need the following result, which is a generalization of Theorem 4.3 of Akhiezer and Cupit-Foutou \cite{ACF}.

\begin{theorem}\label{t:ewv}
Let $k,k_0,\Gamma,G,H,G_0$ be as in \Cref{ss:setup}, and let $Y=G/H$. Assume that:
\begin{enumerate}
\item[(i)] For all $\gamma \in \Gamma$, $\ve_\gamma$ preserves the combinatorial invariants of $Y$;
\item[(ii)] $\sN_G(H)=H$;
\end{enumerate}
Let $\iota \colon Y\into W$ be a wonderful embedding of $Y$. Then $\iota$ admits a unique $G_0$-equivariant $k_0$-form $\iota_0\colon Y_0\into W_0$.
\end{theorem}

\begin{proof}
Since a form is uniquely determined by the semilinear action that it induces, and a semilinear action $\Gamma \to \SAut(W)$ is uniquely determined by its restriction to the (dense) open $G$-orbit, we obtain uniqueness from \Cref{t:Bo}.

Let $Y_0$ be the (unique) $G_0$-equivariant $k_0$-form of $Y$ guaranteed by \Cref{t:Bo}.
Let $\sE_W$ be the colored fan corresponding to the spherical embedding $\iota \colon G/H \into W$.

The $k_0$-form $Y_0$ gives  compatible $\Gamma$-actions on $\X^*(B),\sX(Y),\sV(Y),\sD(Y)$; see Huruguen \cite[Section 2.2]{Huruguen}.
Since the embedding $\iota \colon G/H \into W$ is wonderful, there is a unique maximal colored cone in $\sE_W$, and it consists of $\sV(Y)$ \emph{without colors} (see \Cref{t:wonderfan}).
Since by assumption $\Gamma$ preserves $\sV(Y)$, we see that the colored fan $\sE_W$ corresponding to $\iota \colon G/H \into W$ is $\Gamma$-stable. By \Cref{t:Huruguen}, the semilinear action of $\Gamma$ extends to $W$. Since $W$ is projective, Theorem 2.26 in Huruguen \cite{Huruguen} concludes the proof.
\end{proof}

\section{Rational points in wonderful varieties over large fields}
\label{s:rational}

\begin{lemma}[well-known]
\label{l:Stab}
Let $G$ be an algebraic group over $k$ and let $Y=G/H$ be a homogeneous space.
Let $G_0$ be a $k_0$-form of $G$ with corresponding semilinear action
\[\sigma\colon\Gamma\to\SAut(G).\]
Assume that $Y$ admits a $G_0$-equivariant $k_0$-form $Y_0$.
Then $\sigma_\gamma(H)$ is conjugate to $H$ for every $\gamma\in \Gamma$.
\end{lemma}

\begin{proof}
Denote by $\sigma \colon \Gamma \to \SAut(G)$ the semilinear action of $\Gamma$ on $G$ induced by the $k_0$-form $G_0$, and by $\mu \colon \Gamma \to \SAut(Y)$ the semilinear action of $\Gamma$ on $Y$ induced by the $k_0$-form $Y_0$.

Let $\gamma \in \Gamma$.
Since the form $Y_0$ is $G_0$-equivariant, we see that $\muga$ is $\siga$-equivariant. Thus we see that $\Stab_G(\muga(eH))=\siga(H)$. Choose $g_0\in G(k)$ such that $\muga(eH)=g_0 H$. An elementary calculation shows that $\Stab_G(g_0 H)=g_0H g_0^{-1}$. We see that $\siga(H)=g_0 H g_0^{-1}$ for some $g_0\in G(k)$.
\end{proof}

For a semilinear action $\mu \colon \Gamma \rightarrow \SAut(X)$ of $\Gamma$ on $X$ we will call a ``$\Gamma$-fixed $k$-point in $X$'' a point $x\in X(k)$ such that for all $\gamma \in \Gamma$, $\muga (x)=x$. If $\mu$ arises from a $k_0$-form $X_0$ of $X$, then the $\Gamma$-fixed $k$-points in $X$ are in cannonical bijection with $k_0$-rational points in $X_0$, given by sending a $k_0$-rational point $x\colon \Spec k_0\to X_0$ in $X_0$ to the $\Gamma$-fixed $k$-point $\Spec k\to \Spec k_0\xrightarrow{x} X_0$.

\begin{lemma}\label{l:flag}
Let $k_0,k,\Gamma,G,G_0$ be as in \Cref{ss:setup}.
Let $P\subseteq G$ be a parabolic subgroup, and let $X=G/P$ denote the corresponding generalized flag variety.
Assume that  $G_0$ is {\bf \emph{quasi-split}} and that  $X$ admits a $G_0$-equivariant $k_0$-form $X_0$.
Then $X_0$ has a  $k_0$-point.
\end{lemma}
See Moser-Jauslin and Terpereau \cite[Remark 3.13]{MJT} for a generalization to homogeneous \emph{horospherical} varieties (again, the authors only consider the case $k_0=\R$, but this does not change anything).
\begin{proof}
Let $B_0\subseteq G_0$ be a Borel subgroup defined over $k_0$, and
let $T_0\subseteq B_0$ be a maximal torus.
Set $B=(B_0)_k$, $T=(T_0)_k$.
Consider the set of simple roots $S=S(G,T,B)$.
For $I\subseteq S$, we denote by $P_I$ the standard parabolic subgroup generated
by $B$ and the unipotent subgroups of $G$ associated with the simple roots $\alpha \in I$; see Malle and Testerman \cite[Definition 12.3]{Malle-Testerman}.
Any parabolic subgroup of $G$ is conjugate to  $P_I$ for some $I\subseteq S$.
Moreover, two standard parabolic subgroups $P_I$ and $P_J$ for $I\neq J$ are not conjugate; see Malle and Testerman \cite[Proposition 12.2]{Malle-Testerman}.

Our parabolic subgroup $P$ is conjugate to $P_I$ for some $I\subseteq S$.
Consider the  semilinear action $\sigma \colon \Gamma \rightarrow \SAut(G)$ defined by $G_0$.
Since $G_0$ is quasi-split, it follows from the definition of the $\ast$-action
\[\ve\colon  \Gamma \to\Aut\,\BRD(G,T,B)\]
in \Cref{ss:star-action}
that for all $\gamma \in \Gamma$, we have $\sigma_\gamma(P_I)=P_{\veg(I)}$.
Since $X$ admits a $G_0$-equivariant $k_0$-form, and $P_I$ is the stabilizer of some point in $X$,
by \Cref{l:Stab}  the subgroup $\sigma_\gamma(P_I)$ is conjugate to $P_I$.
It follows that $\veg(I)=I$, and therefore, $\sigma_\gamma(P_I)=P_I$ (for all $\gamma\in\Gamma$).
This $\Gamma$-invariant subgroup $P_I$ corresponds to a $\Gamma$-fixed $k$-point $x_0\in X(k)$.
Clearly, the point $x_0\in X_0(k)=X(k)$ is defined over $k_0$.
\end{proof}

\begin{subsec}\label{ss:Gamma-fixed-point}
Assume now that $W$ is a wonderful $G$-variety with a given $G_0$-equivariant $k_0$-form $W_0$, given by a semilinear action $\mu \colon \Gamma \rightarrow \SAut(W)$.
Denote by $X=\bigcap_{i=1}^{r}X_i$ the unique closed $G$-orbit in $W$.
Note that since $\muga$ is $\siga$-equivariant for all $\gamma \in \Gamma$, $\mu _{\gamma}(X)$ is a closed $G$-orbit for every $\gamma \in \Gamma$. This implies that $X$ has a cannonical $G_0$-equivariant $k_0$-form.
Since $W$ has a \emph{unique} closed $G$-orbit, $X$ is $\Gamma$-invariant.
Since $X$ is a complete homogeneous $G$-variety, it is a generalized flag variety, i.e.
$X=G/P$ for some parabolic subgroup $P\subseteq G$.
By \Cref{l:flag}, if $G_0$ is quasi-split then the variety $X$ contains a $\Gamma$-fixed $k$-point.
\end{subsec}

\begin{theorem}\label{t:fp}
Let $k_0,k,\Gamma,G,G_0$ be as in \Cref{ss:setup}. Let $W$ be a wonderful $G$-variety with a $G_0$-equivariant $k_0$-form $W_0$,
given by a semilinear action $\mu \colon \Gamma \rightarrow \SAut(W)$. Assume that $k_0$ is {\bf \emph{large}}
and that the $k_0$-form $G_0$ is {\bf\emph{quasi-split}}. Then:
\begin{enumerate}
 \item[(i)] Every $\Gamma$-invariant $G$-orbit in $W$ has a $\Gamma$-fixed $k$-point.
 \item[(ii)] The open $G$-orbit in $W$ has a $\Gamma$-fixed $k$-point.
 \item[(iii)] If moreover the $k_0$-form $G_0$ is split, then every $G$-orbit in $W$ is $\Gamma$-invariant.
\end{enumerate}
\end{theorem}

This theorem is a generalization of Theorem 3.6 (ii) of Akhiezer and Cupit-Foutou \cite{ACF}, where the authors considered the case when $k=\R$ and $G_0$ is split.

\begin{proof}
(i) Let $X'$ be a $\Gamma$-invariant $G$-orbit in $W$. Since $X'$ is $\Gamma$-invariant, so is its closure $Z:=\overline{X'}$. Denote by $X$ the (unique) closed $G$-orbit in $W$. Notice that by condition (iii) in \Cref{d:wonderful}, we have $X\subseteq Z$. By \Cref{ss:Gamma-fixed-point},  $X$ contains a $\Gamma$-fixed $k$-point, and hence, $Z$ contains a $\Gamma$-fixed $k$-point.
Since the divisors $X_i$ have normal crossings, and $Z$ is the intersection of some of these divisors, the variety $Z$ is smooth, and hence also reduced (see \Cref{d:wonderful}). Since $Z$ is reduced, and $Z(k)\subseteq W(k)$ is $\Gamma$-invariant, we see that the semilinear action of $\Gamma$ on $W$ restricts to a semilinear action of $\Gamma$ on $Z$, which we denote by $\nu\colon \Gamma \to \SAut(Z)$.
Denote by $\iota$ the embedding of $Z$ into $W$. Since everything is of finite type, there exists a finite extension $k_1/k_0$ in $k$, a $k_1$-scheme $Z_1$ and a closed $k_1$-embedding $\iota_1\colon Z_1\into W_1:=W_0\times_{k_0}k_1$ such that $\iota_1\times_{k_1}k=\iota$. The semilinear action of $\Gamma_1:=\Gal(k_1/k_0)$ on $Z$ induced from $Z_1$ is the same as $\nu|_{\Gamma_1}$, since $Z$ is reduced and on closed points both are given by $\mu$. Since $Z$ can be covered by $\Gamma$-invariant affine open subsets (take such a covering of $W$ and intersect with $Z$), by Borel and Serre \cite[Lemma 6.12]{Borel-Serre}, we see that $\nu$ is induced from a $k_0$-form $Z_0$ of $Z$.
Since $k_0$ is large, and $Z$ is smooth and has a $\Gamma$-fixed $k$-point, we obtain that $\Gamma$-fixed $k$-points (or equivalently: $k_0$-rational points in $Z_0$) are Zariski dense in $Z$.
In particular, the open subset $X'\subseteq Z$ contains a $\Gamma$-fixed $k$-point.

(ii) Let $Y$ be the (unique) open $G$-orbit in $W$. Since $\mu _{\gamma}(Y)$ is an open $G$-orbit for every $\gamma \in \Gamma$, we get that $Y$ is $\Gamma$-invariant, hence $Y$ contains a $\Gamma$-fixed $k$-point by (i).

(iii) Denote by $Y=G/H$ the open $G$-orbit in $W$. Assume now that $G_0$ is split, and denote by $\sE_W$ the colored fan corresponding to the embedding $G/H\into W$. Since $G_0$ is split, $\Gamma$ acts trivially on $\X^*(B)$, and hence preserves $\sX (Y)$. We see that it also acts trivially on $V=\Hom(\sX(Y),\Q)$.
Since $W$ is wonderful, by \Cref{t:wonderfan} there are no colors in the colored fan $\sE_W$. Hence $\Gamma$ preserves each colored cone in $\sE_W$.
From the second assertion of \Cref{t:Huruguen}, we see that every $G$-orbit in $W$ is $\Gamma$-invariant.
\end{proof}

\section{Spherical subgroup defined over the base field}
\label{s:ssubgroup}

The following theorem is a generalization of Theorem 4.4 of Akhiezer and Cupit-Foutou \cite{ACF}, where the authors consider the case when $k_0=\R$ and $G_0$ is split.

\begin{theorem}\label{t:main-qs}
Let $k$ be an algebraically closed field of characteristic $0$, and let $k_0\subseteq k$ be a subfield such that $k/k_0$ is Galois. Denote by $\Gamma = \Gal (k/k_0)$ the Galois group. Let $G$ be a connected reductive $k$-group, $G_0$ a $k_0$-form of $G$. Let $H\subseteq G$ be a spherical subgroup. Assume that $k_0$ is {\bf \emph{large}} and that the $k_0$ form $G_0$ is {\bf\emph{quasi-split}}. Assume moreover that $\charr k_0=0$, and that:
\begin{enumerate}
    \item[(i)] $\Gamma$ preserves the combinatorial invariants of $G/H$ when acting on $\BRD(G)$ via the homomorphism
$\ve\colon\Gamma\to\Aut\,\BRD(G)$  defined by the $k_0$-form $G_0$ of $G$;
    \item[(ii)] $\sN_G(\sN_G(H))=\sN_G(H)$.
\end{enumerate}
Then $H$ is conjugate to a subgroup defined over $k_0$, i.e., there exist $g\in G(k)$ and an algebraic subgroup $H_0\subset G_0$
such that $gHg^{-1}=H_0\times_{k_0} k$.
\end{theorem}

\begin{remark}
When $G_0$ is split, condition (i) is automatically satisfied.
\end{remark}

\begin{proof}
Denote $N=\sN_G(H)$.
Since $\veg$ preserves the combinatorial invariants of $G/H$,
by \Cref{l:Bo} for every $\gamma \in \Gamma$ there exists  $a_\gamma \in G(k)$
such that $\sigma_\gamma(H)=a_\gamma H a_{\gamma}^{-1}$.
Then, of course, we have $\sigma_\gamma(N)=a_\gamma N a_{\gamma}^{-1}$.
By assumption (ii) we have $\sN_G(N)=N$, hence $G/N$ admits a wonderful embedding, say, $G/H\into X$.

Since $\siga (N)$ is conjugate to $N$,  we see from \Cref{l:Bo}
that $\veg$ preserves the combinatorial invariants of $G/N$ for all $\gamma \in \Gamma$.
By \Cref{t:ewv}, $X$ admits a $G_0$-equivariant $k_0$-form $X_0$,
given by a semilinear action $\mu \colon \Gamma \to \SAut(X)$.
By \Cref{t:fp}, there exists a $\Gamma$-fixed $k$-point, say $gN$, in the open orbit  $G/N\subseteq X$.
Then $\mu_{\gamma}(gN)=gN$ for all $\gamma \in \Gamma$.
Define  $N'=gNg^{-1}$ and  $H'=gH g^{-1}$. Then $N'=\Stab_G(gN)$ and $N'=\sN_G(H')$.
For every $\gamma\in \Gamma$, since $\veg$ preserves the combinatorial invariants of $G/H'= G/H$,
there exists $a'_\gamma\in G(k)$ such that $\siga(H')=a'_\gamma\hs H'\hs(a'_\gamma)^{-1}$
(we may take  $a'_\gamma=\siga(g)a_\gamma g^{-1}$).
Clearly we have $\siga(N')=a'_\gamma\hs N'\hs (a'_\gamma)^{-1}$.
We shall show that $\siga(N')=N'$ and $\siga(H')=H'$ for all $\gamma\in\Gamma$.

Let $\gamma \in \Gamma$.
Since $\muga$ is $\siga$-equivariant, we see that
\[\Stab_G(\mu_\gamma(g N))=\siga(\Stab_G(gN))=\siga(N').\]
Since $\mu_\gamma(gN)=gN$, we see that $\siga(N')=N'$.
Since $\siga(N')=a'_\gamma\hs N'\hs(a'_\gamma)^{-1}$, we obtain that
\[a'_\gamma\hs N'\hs(a'_\gamma)^{-1}=N'.\]
This means that $a'_\gamma$ normalizes $N'$.
By assumption $\sN_G(N)=N$, hence $\sN_G(N')=N'$, and therefore, $a'_\gamma\in N'(k)=\sN_G(H')(k)$.
We see that $a'_\gamma$ normalizes $H'$, i.e.
\[a'_\gamma\hs H'\hs(a'_\gamma)^{-1}=H'.\]
Since $a'_\gamma\hs H'\hs (a'_\gamma)^{-1}=\siga(H')$, we obtain that
\begin{equation}\label{e:stable}
\siga(H')=H'.
\end{equation}
Since \eqref{e:stable} holds for all $\gamma\in\Gamma$,  by Galois descent there exists a $k_0$-subgroup $H_0\subseteq G_0$
such that $H_0 \times_{k_0} k =H'$. Thus $gHg^{-1}=H_0 \times_{k_0} k$, as required.
\end{proof}

\begin{example}\label{e:horospherical} Let $k,k_0,G,G_0$ be as in \Cref{t:main-qs}.
Let $H\subseteq G$ be a  \emph{horospherical} subgroup, that is,
$H\supseteq U$, where $U$ is the unipotent radical of a Borel subgroup of $G$.
Then $\sN_G(H)$ is parabolic; see Pasquier \cite[Proposition 2.2]{Pasquier}.
It follows that $\sN_G(H)$ is self-normalizing, and hence, hypothesis (ii) of \Cref{t:main-qs} is satisfied. Thus, for such $H$, and for quasi-split $G_0$, $G/H$ admits a $G_0$-equivariant $k_0$-model if and only if the combinatorial invariants of $G/H$ (or, equivalently in this case, the \emph{horospherical datum} of $G/H$; see \cite[Definition 3.6]{MJT}) are preserved under the $*$-action (see Moser-Jauslin and Terpereau \cite[Corollary 3.12]{MJT} for this result in the case where $k_0=\R$).
This happens for example for $H=U$ (This can be seen easily: we simply take $U_0 = [B_0 , B_0]$). See \Cref{e:horospin} below.
\end{example}

\begin{example}\label{e:symmetric} Let $k,k_0,G,G_0$ be as in \Cref{t:main-qs}.
In  Moser-Jauslin and Terpereau \cite[The proof of Corollary 2.7]{MJT2} the authors showed that if $H\subseteq G$ is \emph{symmetric} (i.e., there exists an involutive $k$-automorphism $\theta$ of $G$ such that $G^\theta \subseteq H \subseteq \sN_G(G^\theta)$), then condition (ii) in the theorem above is satisfied. Thus for such an $H$, the homogeneous variety $G/H$ will admit a $G_0$-equivariant $k_0$-form if and only if for every $\gamma \in \Gamma$, $\sigma_\gamma (H)$ will be conjugate in $G$ to $H$. This happens for example if $G_0$ is split. See \Cref{e:Trialitary} below.
\end{example}

\begin{example} Let $k,k_0,G$ be as in \Cref{t:main-qs}.
Let $G_0$ be a split $k_0$-form of $G$, then hypothesis (i) is satisfied. Thus, if $H$ has self-normalizing normalizer, which happens e.g., if $H\subseteq G$ is horospherical (see \Cref{e:horospherical}), or if $H\subseteq G$ is symmetric (see \Cref{e:symmetric}), then the subgroup $H$ is conjugate to a subgroup of $G_0$ defined over $k_0$.
Hence, in this case, $G/H$ admits a $G_0$-equivariant $k_0$-form.
\end{example}


\begin{example}\label{e:Trialitary}
 Let $k_0$ be a $p$-adic field, that is, a finite extension of the field of $p$-adic numbers  $\Q_p$. Let $G=\Spin_8$ over $k=\overline{k_0}$, and let $G_0$ be a   \emph{trialitary} $k_0$-form of $G$, i.e., a $k_0$-form of $G$ where the image of the Galois group in the group of automorphisms of the Dynkin
 diagram is isomorphic to  $\Z/3\Z$  or $S_3$. 
 Then $G_0$ is quasi-split.
 Indeed, $G_0$ is an inner twisted form of a quasi-split trialitary $k_0$-form $G_\qs$ of $G$:
 we have $G_0={}_c G_\qs$, where $c\in Z^1(k_0,G_\qs/Z(G_\qs))$.
 By \Cref{l:H1-H2} below we have $H^1(k_0,G_\qs/Z(G_\qs))=1$.
 It follows that $[c]=1\in H^1(k_0,G_\qs)=1$, and hence, $G_0\simeq G_\qs$.

 Let $H=\Spin(4)\cdot \Spin(4)\subset \Spin(8)$.
Then $H$ is a semisimple subgroup of type $\AA_1^4$, that is, $H$ is isogenous to $\SL(2)^4$.
Let $\gamma\in\Gamma=\Gal(k/k_0)$, and let $\sigma_\gamma\colon G\to G$
denote the corresponding $\gamma$-semilinear automorphism of $G$.
Then $\sigma_\gamma(H)$ is again a semisimple subgroup of type $\AA_1^4$.
An easy argument (see \Cref{l:A1-4} below) shows that any two subgroups of type $\AA_1^4$
in a group of type $\DD_4$ are conjugate.
Thus $\sigma_\gamma(H)$ is conjugate to $H$ for all $\gamma\in\Gamma$.
Since $H\subset G$ is symmetric, by \Cref{r:Ronan}, $H$ has self-normalizing normalizer. By \Cref{t:main-qs}, $G/H$ admits a $G_0$-equivariant $k_0$-form.
\end{example}

I thank Mikhail Borovoi for help with proving \Cref{l:H1-H2,l:A1-4} below.

\begin{lemma}[well-known]
\label{l:H1-H2}
Let $G_\qs$ be a quasi-split  {\em trialitary} $k_0$-form of a simply connected algebraic $k$-group $G$ of type $\DD_4$,
where $k_0$ is a $p$-adic field.
Write $G_\qs^\ad=G_\qs/Z_\qs$, where $Z_\qs=Z(G_\qs)$.
Then $H^1(G_\qs^\ad)=1$.
\end{lemma}

\begin{proof}
We compute $H^2(k_0,Z_\qs)$.
By the Tate-Nakayama duality (see Serre \cite[Section II.5.2, Theorem 2]{Serre}),
$H^2(k_0,Z_\qs)$ is dual to the group of fixed points $\X^*(Z)^\Gamma$, where $Z=Z(G)$ and the Galois action on $\X^*(Z)$ is induced from the $k_0$-form $Z_\qs$ of $Z$.
It is clear that $\X^*(Z)^\Gamma=1$ (notice that $Z(k)\cong\Z_2\times \Z_2$, and the Galois action is transitive on $Z(k)\backslash\{e\}$), hence $H^2(k_0,Z_\qs)=1$.

Now we compute $H^1(k_0,G_\qs^\ad)$.
The short exact sequence
\[1\to Z_\qs\to G_\qs\to G^\ad_\qs\to 1\]
gives rise to a cohomology exact sequence
\[ 1=H^1(k_0,G_\qs)\to H^1(k_0,G^\ad_\qs)\to H^2(k_0,Z_{\qs})=1,\]
in which $H^1(k_0,G_\qs)=1$ by Kneser's theorem; see Platonov and Rapinchuk \cite[Theorem 6.4]{Platonov-Rapinchuk}.
We conclude that $H^1(k_0,G^\ad_\qs)=1$, as required.
\end{proof}

\begin{lemma}[well-known]
\label{l:A1-4}
Any two connected $k$-subgroups $H,H'$   of type $\AA_1^4$
in a simple $k$-group of type $\DD_4$ are conjugate.
\end{lemma}

\begin{proof}
Let $T,T'$ be  maximal tori in $H,H'$ respectively; then $T$ and $T'$ are maximal tori in $G$.
Up to conjugacy, we may assume that $T'=T$.
Thus $H,H'\supset T$.
We see that $H$  defines a subset $\Theta=R(H,T)\subset R(G,T)=R$ with the following properties:
(a) if $\beta\in\theta$, then $-\beta\in \Theta$, and (b) there are four pairwise orthogonal roots
$\beta_0,\beta_1,\beta_3,\beta_4$ in $\Theta$.
Similarly,  $H'$  defines a subset $\Theta'=R(H',T)\subset R)$  with the same properties.
It suffices to show that there exists an element $w$ of the Weyl group $W(R)=W(G,T)$
sending $\Theta$ to $\Theta'$.

We write $R$ as in Bourbaki \cite[Plate IV]{Bourbaki}:
\[R=\{\pm\ve_i\pm\ve_j\mid 1\le i<j\le 4\}.\]
We set
\[\Phi=\{\alpha_0=\ve_1+\ve_2,\ \alpha_1=\ve_1-\ve_2,\  \alpha_3=\ve_3-\ve_4,\ \alpha_4=\ve_3+\ve_4\}.\]
Since $W$ acts on $R$ transitively, there exists an element $w\in W$ sending $\beta_0$ to $\alpha_0$.
Up to $W$-conjugacy we may assume that $\beta_0=\alpha_0$.
Up to sign, we may assume that $\beta_1,\,\beta_3,\,\beta_4$ are positive roots.
Thus $\beta_1,\,\beta_3,\,\beta_4$ are three positive roots in $R$ which are orthogonal to $\alpha_0$.
Now it is clear that the set $\{\beta_1,\,\beta_3,\,\beta_4\}$ coincides with the set $\{\alpha_1,\,\alpha_3,\,\alpha_4\}$.
It follows that $\Theta=\Phi$.

We have shown that $\Theta$ is $W$-conjugate to $\Phi$.
Similarly, $\Theta'$ is $W$-conjugate to $\Phi$.
Thus $\Theta$ and $\Theta'$ are $W$-conjugate, and hence, $H$ and $H'$ are conjugate in $G$.
\end{proof}

\begin{example}\label{e:horospin}
Let $k,G,k_0,G_0$ be as in \Cref{e:Trialitary}. Let $H\subseteq G$ be a horospherical subgroup with horospherical datum $(I,M)$; see Moser-Jauslin and Terpereau \cite[Definition 3.6]{MJT}. Here $I\subseteq S$ is a subset of the set of simple roots, and $M\subseteq\X ^*(B)$ is a sublattice whose elements satisfy $\langle \chi,\alpha\rangle=0$ for all $\alpha \in I$ (i.e., characters admitting an extension to $P_I$). Then, since $\sN_G (\sN_G(H))=\sN_G(H)$, by \Cref{t:main-qs} and Pasquier \cite[Proposition 2.4]{Pasquier}, $G/H$ admits a $G_0$-equivariant $k_0$-form if and only if $I,M$ are invariant under the $*$-action.
Here are some examples of such $(I,M)$:\\
In the notations of \Cref{l:A1-4} (with $\alpha_2=\epsilon_2-\epsilon_3$), we can take e.g. $I=\{\alpha_2\}$ or $I=\{\alpha_1,\alpha_3,\alpha_4\}$. In the second case, any $M\subseteq \Span_\Z \{\alpha_1,\alpha_3,\alpha_4\}^\perp=\Span_\Z\{\frac{\epsilon_1+\epsilon_2}{2}\}=\Span_\Z\{\frac{\alpha_1+2 \alpha_2+\alpha_3 +\alpha_4}{2}\}$ would be $\Gamma$-invariant.
Here are some examples of $\Gamma$-invariant $M$ for the case $I=\{\alpha_2\}$:
\begin{enumerate}
\item[(i)] $M_1=n \cdot \Span_\Z\{\epsilon_1+\frac{\epsilon_2+\epsilon_3}{2}\}=n\cdot \Span_\Z\{\frac{3\alpha_2}{2}+\alpha_1+\alpha_3+\alpha_4\}$
\item[(ii)] $M_2=m\cdot \Span_\Z\{\frac{\epsilon_2+\epsilon_3-\epsilon_1-\epsilon_4}{2},\epsilon_4\}=m\cdot \Span_\Z\{\frac{\alpha_1-\alpha_3}{2},\frac{\alpha_3-\alpha_4}{2}\}$
\item[(iii)] $M_4=M_1\oplus M_2$
\item[(iv)] $M_5=n\cdot \Ker(\alpha_2^\vee)$
\end{enumerate}
Here, $n,m$ are integers.
\end{example}



\section{Existence of a form over a large field}
\label{s:final}

Let $k,k_0,G,G_0$ be as in \Cref{ss:setup}.
Write $\Gbar=G/Z(G)$ for the corresponding adjoint group, and $\Gtil$ for the universal covering of $G'=[G,G]$.
By \Cref{p:inn-qs} we may write $G_0=\hs _c(G_\qs)$,
where $G_\qs$ is a quasi-split $k_0$-form of $G$ and $c\in Z^1(k_0, \Gbar_\qs)$.
Here we write $\Gbar_\qs=G_\qs/Z(G_\qs)$.

\begin{subsec}\label{ss:Tits}
We write $\Ztil_\qs$ for the center $Z(\Gtil_\qs)$ of the universal cover $\Gtil_\qs$ of the connected semisimple group
$[G_\qs,G_\qs]$.
Note that $\Ztil_\qs=\Ztil_0$.
The short exact sequence
\[1\to \Ztil_\qs\to \Gtil_\qs \to \Gbar_\qs\to 1\]
induces a cohomology exact sequence
\[H^1(k_0,\Ztil_\qs)\to  H^1(k_0,\Gtil_\qs)\to H^1(k_0,\Gbar_\qs)\labelto{\Delta} H^2(k_0,\Ztil_\qs).\]
By definition, the \emph{Tits class of} $\Gtil_0$  is the image of $[c]\in Z^1(k_0,\Gbar_\qs)$ in $H^2(k,\Ztil_\qs)$
under the connecting map $\Delta\colon H^1(k_0,\Gbar_\qs)\to{\Delta} H^2(k_0,\Ztil_\qs)$; see \cite{KMRT}, Section 31, before Proposition 31.7.
\end{subsec}

\begin{subsec} \label{ss:Aqs}
Let $H\subset G$ be a spherical subgroup. Set $A=\sN_G(H)/H$, which is a group multiplicative type (i.e. a closed subgroup of a torus);
see Losev \cite[Theorem 2 and Definition 4.1.1(1)]{Losev}.
The  character group $\X^*(A)$ of $A$ is a quotient group of the weight lattice $\sX$ of $Y=G/H$.
The group $\X^*(A)$ and the surjective homomorphism $\sX\to\X^*(A)$ can be computed from the combinatorial invariants of $G/H$; see Losev \cite[Theorem 2 and Definition 4.1.1(1)]{Losev} again.

We use the notation of \Cref{ss:star-action}.
In particular, for any $\gamma\in\Gamma$ we have an automorphism $\veg\in\Aut\,\BRD(G,T,B)$.
Assume that  $\veg$ preserves the combinatorial invariants of $G/H$ for all $\gamma\in \Gamma$.
Then $\veg$ naturally acts on $\sX$, on $\X^*(A)$, and on $\X^*(\Ztil)$, and the homomorphisms
\[ \sX\to\X^*(A) \quad\text{and}\quad \X^*(A)\to\X^*(\Ztil)\]
are $\Gamma$-equivariant.
Thus we obtain a $k_0$-form of $A$, which we denote $A_\qs$, and a $k_0$-homomorphism
$\Ztil_\qs\to A_\qs$.
\end{subsec}

\begin{theorem}\label{t:main-general}
Let $k$ be an algebraically closed field of characteristic $0$, and let $k_0\subseteq k$ be a subfield such that $k/k_0$ is Galois. Denote by $\Gamma = \Gal (k/k_0)$ the Galois group. Let $G$ be a connected reductive $k$-group, $G_0$ a $k_0$-form of $G$. Let $H\subseteq G$ be a spherical subgroup.
Assume that:\\
(i)  $\Gamma$ preserves the combinatorial invariants of $G/H$ when acting on $\BRD(G)$ via the homomorphism
$\ve\colon\Gamma\to\Aut\,\BRD(G)$  defined by the $k_0$-form $G_0$ of $G$;\\
(ii) $k_0$ is a large field;\\
(iii) $\charr k_0=0$;\\
(iiii) $\sN_G(\sN_G(H))=\sN_G(H)$.\\
Then $Y=G/H$ admits a $G_0$-equivariant $k_0$-form if and only if the image in $H^2(k_0,A_\qs)$ of the Tits class $t(\Gtil_0)\in H^2(k_0,Z(\Gtil_0))$ is 1.
\end{theorem}

Note that assumption (i) is a necessary condition for $G/H$ to have a $k_0$-form; see Borovoi \cite[Proposition 8.12]{BG}.

\begin{proof}
The $k_0$-forms $G_0$ and $G_\qs$ of $G$ define the same homomorphism
\[ \ve^0=\ve^\qs\colon\, \Gamma\to\Aut\,\BRD(G).\]
Since $\ve^0$ preserves the combinatorial invariants of $G/H$, so does $\ve^\qs$.
Since $G_\qs$ is quasi-split,
by \Cref{t:main-qs} the subgroup $H$ is conjugate to some subgroup defined over $k_0$ in $G_\qs$, say $H_\qs\times_{k_0} k$,
where $H_\qs\subset G_\qs$.
Then clearly $G_\qs/H_\qs$ is a $G_\qs$-equivariant $k_0$-form of $G/H$.
Recall that $G_0=\hs_c (G_\qs)$, where $c\in Z^1(k_0, \Gbar_\qs)$.
By Borovoi and Gagliardi \cite[Theorem 3.8]{BG2}
the homogeneous variety $G/H$ admits a $G_0$-equivariant $k_0$-form if and only if
the image  of  $t(\Gtil_0)$  in $H^2(k_0,A_\qs)$ is 1, as required.
\end{proof}

\begin{example}
Let $G=\SL_{2,k}$, and let $H=T$ or $H=\sN_G(T)$. Then $G/H$ admits a $G_0$-equivariant $k_0$-form (for any $k_0$ and any $G_0$). Indeed, for any $k_0$-form $G_0$ of $G$ (and any field $k_0$), $G_0$ has a maximal torus $T_0$ defined over $k_0$, and thus we obtain a $k_0$-form $G_0/T_0$ of $G/T$.
\end{example}

See \cite{BG2} for many more interesting non-quasi-split examples, as well as a generalization of the main theorems presented in this paper to the case where $k_0$ is an arbitrary field of characteristic 0, and $H\subseteq G$ is an arbitrary spherical subgroup.

\end{document}